\documentclass[12pt,reqno]{amsart}
\usepackage{amsmath}
\usepackage{amssymb}
\usepackage{amstext}
\usepackage{a4wide}
\usepackage{txfonts}
\usepackage{graphicx}
\usepackage{bm}
\usepackage[numbers,sort&compress]{natbib}
\allowdisplaybreaks \numberwithin{equation}{section}
\usepackage{color}
\usepackage{cases}

\numberwithin{equation}{section}

\newtheorem{theorem}{Theorem}[section]
\newtheorem{proposition}[theorem]{Proposition}
\newtheorem{corollary}[theorem]{Corollary}
\newtheorem{lemma}[theorem]{Lemma}

\newtheorem*{Yudovich's Theorem}{Yudovich's Theorem}

\theoremstyle{definition}

\newtheorem{definition}[theorem]{Definition}

\theoremstyle{remark}
\newtheorem{remark}[theorem]{Remark}

\newtheorem{example}[theorem]{Example}

\begin{document}

\title[Orbital Stability of First Laplacian Eigenstates for the Euler Equation]{Orbital Stability of First Laplacian Eigenstates for the Incompressible Euler Equation on a Flat 2-Torus}
 \author{Guodong Wang}

\address{School of Mathematical Sciences, Dalian University of Technology, Dalian 116024, PR China}
\email{gdw@dlut.edu.cn}

\begin{abstract}
On a two-dimensional flat torus, Laplacian eigenfunctions admit explicit trigonometric representations. It is known that every first eigenstate on a rectangular or square torus is stable under the incompressible Euler dynamics modulo translations. We extend this result to flat tori of arbitrary shape and thereby obtain, to the best of our knowledge, the first family of orbitally stable sinusoidal Euler flows on a hexagonal torus. The proof uses a Burton-type stability criterion and has two main ingredients: (i) a variational characterization of each equimeasurable class in the first eigenspace and (ii) the finiteness of the number of translational orbits contained in each such class. The second ingredient is particularly delicate in the hexagonal case, where it reduces to the analysis of a polynomial system reflecting both the symmetry of the torus and the structure of its first eigenspace.
\end{abstract}


\renewcommand{\subjclassname}{2020 Mathematics Subject Classification}

\maketitle
\tableofcontents
\section{Introduction}\label{sec1}

\subsection{The Euler equation on a flat 2-torus}\label{sec11}

A flat $2$-torus is the quotient of the Euclidean plane by a two-dimensional lattice.
Throughout this paper, let $\Lambda$ be a two-dimensional lattice generated by two linearly independent vectors $\bm{\xi}, \bm{\eta} \in \mathbb{R}^2$, i.e.,
\[
\Lambda = \left\{ m\bm{\xi} + n\bm{\eta} \mid m,n \in \mathbb{Z} \right\}.
\]
The pair $(\bm{\xi}, \bm{\eta})$ is called a \emph{basis} of $\Lambda$.
Note that the basis of a lattice is not unique.
Denote by $\mathbb{T} = \mathbb{R}^2 / \Lambda$ the flat $2$-torus associated with $\Lambda$.
When $\bm{\xi}$ and $\bm{\eta}$ vary, we obtain flat $2$-tori of different shapes. All flat $2$-tori have the same topology;  however, their global geometries can differ, which may lead to notable differences in certain problems, such as the number of critical points of the Green function (see \cite{LW}).

For an ideal (i.e., incompressible and inviscid) fluid of unit density on $\mathbb T$, the evolution is governed by the following Euler equation:
\begin{equation}\label{euler}
\begin{cases}
  \partial_t\bm v+(\bm v\cdot\nabla)\bm v=-\nabla P, &t\in\mathbb R,\,\bm x=(x_1,x_2)\in\mathbb T,\\
  \nabla\cdot\bm v=0,
\end{cases}
\end{equation}
where $\bm v=(v_1,v_2)$ is the velocity field, and $P$ is the scalar pressure.
The study of the Euler equations on $\mathbb T$ is equivalent to that of the corresponding equations on $\mathbb R^2$ subject to the doubly periodic conditions
\[
\bm v(t,\bm x)=\bm v(t,\bm x+\bm\xi),\quad
\bm v(t,\bm x)=\bm v(t,\bm x+\bm\eta),\quad\forall\,t \in \mathbb{R},\ \bm{x} \in \mathbb{R}^2.\]
Such doubly periodic conditions arise naturally in the study of vortex arrays and point-vortex dynamics; see, for example, \cite{SA,TK}. Although every flat $2$-torus can be mapped linearly onto a rectangular one, such a map does not, in general, preserve the Euclidean metric, the Laplacian, or the kinetic energy. Therefore, the shape of the torus is an essential feature of the Euler problem considered here.

Since the integral of the velocity is a conserved quantity (see Lemma \ref{lma1} in Appendix \ref{apdx1}), we may assume,
up to a Galilean transformation,  that $\bm v$ has zero mean.
Introduce the scalar vorticity $\omega:=\partial_1v_2-\partial_2v_1$, which automatically has zero mean. Denote by $\mathsf G$  the inverse of $-\Delta$ on $\mathbb T$ subject to the mean-zero condition; see Definition \ref{defap1}.
According to Lemma \ref{lma2}, we have
\[
\bm{v} = \nabla^\perp \mathsf G\omega,
\quad \nabla^\perp := (\partial_2, -\partial_1).\]
 Therefore,
the Euler equation \eqref{euler} can be rewritten as follows:
\begin{equation}\label{voreq}
\partial_t\omega+\nabla^\perp\mathsf G\omega \cdot\nabla\omega=0,\quad t\in\mathbb R,\,\, \bm x \in\mathbb T.
\end{equation}
 There are many global well-posedness results for \eqref{voreq} with initial vorticity in various function spaces; see \cite{BKM,De,DM,Y}. In particular, given a smooth mean-zero function $\omega_0$ on $\mathbb T$, there exists a unique global smooth mean-zero  solution $\omega$ such that $\omega(0,\cdot)=\omega_0.$

For sufficiently smooth solutions to \eqref{voreq}, the following two conservation laws hold (see \cite{BV,MB,MP}):
\begin{itemize}
  \item [(C1)]
  The kinetic energy $E$   is conserved, where $E$ is regarded as a functional of $\omega$ in this paper:
  \begin{equation}\label{s05}
    E(\omega)=\frac{1}{2}\int_{\mathbb T} \omega\mathsf G\omega d\bm x.
  \end{equation}
  \item [(C2)]  The distribution function of the vorticity is invariant:
       \[\omega(t,\cdot)\in\mathcal R_{\omega(0,\cdot)}\quad\forall\, t\in\mathbb R,\]
      where $\mathcal R_f$ denotes the  rearrangement class of a given measurable function $f$, i.e.,
 \[
 \mathcal R_{f}=\left\{ g:\mathbb T\to\mathbb R \mid |\{\bm  x\in\mathbb T\mid g(\bm   x)>s\}|=|\{\bm   x\in\mathbb T\mid f(\bm   x)>s\}|\,\,\forall\,s\in\mathbb R\right\},
 \]
where $|\cdot|$ denotes the two-dimensional Lebesgue measure. As a consequence, there exist infinitely many integral invariants, known as \emph{Casimirs}, of the form
      $\int_{\mathbb T}F(\omega)d\bm x,$
      where $F:\mathbb R\to\mathbb R$ is any Borel measurable function.
 \end{itemize}

A steady solution to \eqref{voreq} is a solution not depending on the time variable. It is clear that $\bar\omega:\mathbb T\to\mathbb R$ is a steady solution if and only if   $\nabla\mathsf G\bar\omega$ and $\nabla\bar\omega$ are parallel. In particular, if $u\in C^2(\mathbb T)$ satisfies
\begin{equation}\label{sscon}
\begin{cases}
-\Delta u=\varphi(u),&\bm x\in\mathbb T,\\
\int_{\mathbb T}u d\bm x=0
\end{cases}
\end{equation}
for some $\varphi\in C^1(\mathbb R)$, then $\bar\omega=-\Delta u$ is a steady solution. In the literature, there are many results on the construction and classification of steady solutions to the two-dimensional Euler equation in $\mathbb R^2$ or in domains with a boundary; however, on a flat 2-torus, such results are rather scarce. Recently, Elgindi and Huang \cite{EH} proved the existence of both smooth and singular steady solutions around the Bahouri–Chemin patch on a square torus by studying \eqref{sscon} for suitable choices of  $\varphi$. It is not clear whether their construction remains valid on a flat 2-torus of arbitrary shape.

\subsection{Dual lattice and Laplacian eigenfunctions}\label{sec12}
Choosing $\varphi(s)=\lambda s$ in \eqref{sscon},
we obtain the following Laplacian eigenvalue problem:
\begin{equation}\label{evp}
\begin{cases}
 -\Delta u=\lambda u, & \bm  x\in \mathbb T, \\
 \int_{\mathbb T}ud\bm x=0.
\end{cases}
\end{equation}
To solve \eqref{evp}, we define the dual lattice $\Lambda^*$ of $\Lambda$ as follows:
\[\Lambda^*=\left\{\bm k\in\mathbb R^2\mid \bm  k\cdot \bm \xi\in\mathbb Z, \bm k\cdot\bm \eta\in\mathbb Z\right\}.\]
A basis  $(\bm \xi^*,\bm\eta^*)$ of $\Lambda^*$  can be computed as follows (as one can easily verify):
\begin{equation}\label{dulbas}
\bm\xi^*=\frac{(\eta_2,-\eta_1)}{\xi_1\eta_2-\xi_2\eta_1},\quad \bm\eta^*
=\frac{(-\xi_2,\xi_1)}{\xi_1\eta_2-\xi_2\eta_1}.
\end{equation}
Note that such a basis satisfies
\begin{equation}\label{dulbas2}
\bm\xi^*\cdot\bm\xi=1,\quad \bm\xi^*\cdot\bm\eta=\bm\xi\cdot\bm\eta^*=0,\quad \bm\eta^*\cdot\bm\eta=1.
\end{equation}
According to Lemma \ref{lemeigen1} in Section \ref{sec2}, the set of  eigenvalues of \eqref{evp} is
\begin{equation}\label{evevp}
\left\{4\pi^2|\bm k|^2\mid \bm k\in\Lambda^*\setminus\{(0,0)\}
\right\},
\end{equation}
and the eigenspace associated with an eigenvalue $\lambda$ is
\begin{equation}\label{efevp}
{\rm span}\left\{e^{2\pi {\rm i}\bm k\cdot \bm x}\mid 4\pi^2|\bm k|^2 =\lambda\right\},
\end{equation}
 where ${\rm i}^2=-1$. Note that different vectors \(\bm k\) may yield the same eigenvalue, so the eigenvalues may have nontrivial multiplicities.

In this paper, we will focus on the first eigenvalue $\lambda_1$
and the first eigenspace $\mathbb E_1$. Denote by $\rho(\Lambda^*)$ the shortest nonzero length of $\Lambda^*$:
\[
\rho(\Lambda^*) = \min_{\bm k \in \Lambda^* \setminus \{(0,0)\}} |\bm k|,
\]
and by $S(\Lambda^*)$ the set of shortest nonzero vectors in $\Lambda^*$:
\[
S(\Lambda^*) = \{\bm k \in \Lambda^* \mid |\bm k| = \rho(\Lambda^*)\}.
\]
According to \eqref{efevp} and \eqref{evevp},
\[
\lambda_1 = 4\pi^2 \rho(\Lambda^*)^2, \quad \mathbb E_1 ={\rm span} \left\{ e^{2\pi {\rm i} \bm k \cdot \bm x} \mid \bm k \in S(\Lambda^*) \right\}.
\]
It is clear that
\begin{equation}\label{numbs}
\dim(\mathbb E_1) = \# S(\Lambda^*).
\end{equation}
A detailed discussion of the dimension of $\mathbb{E}_1$ is provided in Lemma~\ref{lemdime1} in Section \ref{sec21}.

For the reader’s convenience, we present three representative examples below.

\begin{example}[Rectangular torus]\label{examp1}
Let
\[
\mathbb T=\mathbb R^2/\Lambda,\quad
\Lambda = \{ m\bm\xi + n\bm\eta \mid m,n \in \mathbb{Z} \}, \quad
\bm\xi = 2\pi (1,0), \quad
\bm\eta = h(0,1),
\]
where $0<h<2\pi.$
According to \eqref{dulbas},
\[
\Lambda^* = \{ m\bm\xi^* + n\bm\eta^* \mid m,n \in \mathbb{Z} \}, \quad
\bm\xi^* = \frac{1}{2\pi} (1,0), \quad
\bm\eta^* = \frac{1}{h} (0,1).
\]
It is clear that
\[\rho(\Lambda^*)=\frac{1}{2\pi},\quad S(\Lambda^*)=\left\{\pm \bm \xi^*\right\}.\]
Hence
\[ \lambda_1=1,\quad
\mathbb E_1 ={\rm span}\left\{\cos x_1,\sin x_1\right\}.
\]
\end{example}

\begin{example}[Square torus]\label{examp2}
Let
\begin{equation}\label{stll}
\mathbb T=\mathbb R^2/\Lambda,\quad \Lambda = \{ m\bm\xi + n\bm\eta \mid m,n \in \mathbb{Z} \}, \quad
\bm\xi = 2\pi (1,0), \quad \bm\eta = 2\pi (0,1).
\end{equation}
According to \eqref{dulbas},
\[
\Lambda^* = \{ m\bm\xi^* + n\bm\eta^* \mid m,n \in \mathbb{Z} \}, \quad
\bm\xi^* = \frac{1}{2\pi} (1,0), \quad \bm\eta^* = \frac{1}{2\pi} (0,1).
\]
It is clear that
\[\rho(\Lambda^*)=\frac{1}{2\pi},\quad S(\Lambda^*)=\left\{\pm\bm \xi^*, \pm\bm \eta^*\right\}.\]
Hence
\[ \lambda_1=1,\quad \mathbb E_1={\rm span}\left\{\cos x_1,\sin x_1,\cos x_2,\sin x_2\right\}.\]
\end{example}

\begin{definition}[Hexagonal torus]\label{defet}
  If $\Lambda$  has a basis  $(\bm{\xi},\bm{\eta})$ such that $\bm{\xi},\bm{\eta}$ have equal lengths and form an angle of $\pi/3$, then $\Lambda$ is called a \emph{hexagonal lattice}; accordingly, $\mathbb T$ is called a \emph{hexagonal torus}.
\end{definition}

\begin{example}[Hexagonal torus]\label{examp3}
Let
\begin{equation}\label{eisto1}
  \mathbb T=\mathbb R^2/\Lambda,\quad  \Lambda=\{m\bm\xi+n\bm\eta\mid m,n\in\mathbb Z\},\quad \bm\xi=2\pi(1,0),\quad \bm\eta=2\pi\left(\frac{1}{2},\frac{\sqrt{3}}{2}\right).
    \end{equation}
According to \eqref{dulbas},
\[
\Lambda^*=\{m\bm\xi^*+n\bm\eta^*\mid m,n\in\mathbb Z\},\quad \bm\xi^*=\frac{1}{2\pi}\left(1,-\frac{1}{\sqrt{3}}\right),\quad \bm\eta^*=\frac{1}{2\pi}\left(0,\frac{2}{\sqrt{3}}\right).
\]
It is clear that
\[\rho(\Lambda^*)=\frac{1}{\sqrt{3}\pi},\quad S(\Lambda^*)=\left\{\pm\bm\xi^*, \pm\bm\eta^*, \pm(\bm\xi^*+\bm\eta^*)\right\}.\]
Hence $\lambda_1=4/3$, and $\mathbb E_1$ is spanned by the following six functions:
\[
 \cos\left(x_1-\frac{x_2}{\sqrt{3}}\right),
\sin\left(x_1-\frac{x_2}{\sqrt{3}}\right),
\cos\left(\frac{2 x_2}{\sqrt{3}}\right),
\sin\left( \frac{2 x_2}{\sqrt{3}}\right),  \cos\left(x_1+\frac{x_2}{\sqrt{3}}\right),
\sin\left(x_1+\frac{x_2}{\sqrt{3}}\right).
\]
\end{example}

The streamlines of the first eigenstates on a hexagonal torus can be very different from those on a rectangular or square torus. For example, on the hexagonal torus \eqref{eisto1}, the eigenfunction
\begin{equation}\label{sdpt}
\cos\left(x_1-\frac{x_2}{\sqrt{3}}\right)+\cos\left(\frac{2 x_2}{\sqrt{3}}\right)+\cos\left(x_1+\frac{x_2}{\sqrt{3}}\right)
\end{equation}
 has one maximum point, two minimum points, and three saddle points, and the corresponding flow contains one large positive vortex and two small negative vortices. In contrast, on the square torus \eqref{stll}, the eigenfunction
$\cos x_1 +\cos x_2$
has one maximum point, one minimum point, and two saddle points, and the corresponding flow contains two opposite-signed vortices of equal size.

\subsection{Main theorem}\label{sec13}

Throughout this paper, let \(1 < p < \infty\) be fixed.
For convenience,  we place a small circle above  a given function space to denote its subspace  of  mean-zero functions; for example,
\begin{equation}\label{mr1}
\mathring L^p(\mathbb T)=\left\{f\in L^p(\mathbb T) \Bigm| \int_{\mathbb T}f d\bm x=0\right\},\quad \mathring W^{2,p}(\mathbb T)=\left\{f\in W^{2,p}(\mathbb T) \Bigm| \int_{\mathbb T}f d\bm x=0\right\}.
\end{equation}

To make our stability result more general, we first introduce the notion of $L^p$-admissible map.
\begin{definition}[$L^p$-admissible map]\label{defam}
If \(\zeta \in C(\mathbb{R}; \mathring{L}^p(\mathbb{T}))\) satisfies
\[
E(\zeta(t)) = E(\zeta(0)),\footnote{The functional $E$ is well defined on $\mathring{L}^p(\mathbb{T})$; see Appendix \ref{apdx1}.} \quad \zeta(t) \in \mathcal{R}_{\zeta(0)}
\]
for any \(t \in \mathbb{R}\), then \(\zeta\) is called an \emph{$L^p$-admissible map}.
\end{definition}

By (C1) and (C2) in the previous subsection, \(\zeta(t) := \omega(t, \cdot)\) is an $L^p$-admissible map for any sufficiently smooth solution \(\omega\) of the Euler equation \eqref{voreq}.

The main theorem of this paper is as follows.

\begin{theorem}\label{thm1}
Every $\bar\omega\in\mathbb E_1$ is stable up to translations in the following sense:  for any  $\varepsilon>0,$ there exists some $\delta>0,$ such that for any $L^p$-admissible map $\zeta(t)$ in the sense of Definition \ref{defam}, if
\[\left\|\zeta(0)-\bar\omega\right\|_{L^p(\mathbb T)}<\delta,\]
then for any $t\in\mathbb R$, there exists some $\bm p\in \mathbb R^2$ such that
\[ \left\|\zeta(t)-\bar\omega(\cdot-\bm p)\right\|_{L^p(\mathbb T)}<\varepsilon.\]
\end{theorem}

 Here and throughout, points in $\mathbb R^2$ are always understood modulo $\Lambda.$

\begin{remark}
  For a rectangular or square torus, Theorem \ref{thm1} has  been proved in \cite{WZCV}.
\end{remark}

\begin{remark}\label{rmobt}
 Denote by $\mathcal O_{\bar\omega}$ the orbit of $\bar\omega$ under the action of the translation group, i.e.,
 \begin{equation}\label{tobt}
 \mathcal O_{\bar\omega}=\left\{\bar\omega(\cdot-\bm p) \mid \bm p\in \mathbb R^2\right\}.
 \end{equation}
 It is clear that $\mathcal O_{\bar\omega}$ is compact in $\mathring L^p(\mathbb T)$. The conclusion of Theorem \ref{thm1} can then be reformulated as follows: for any  $\varepsilon>0,$ there exists some $\delta>0,$ such that for any $L^p$-admissible map $\zeta(t)$ in the sense of Definition \ref{defam}, it holds that
\begin{equation}\label{sosv}
\min_{f\in\mathcal O_{\bar\omega}}\left\|\zeta(0)-f\right\|_{L^p(\mathbb T)}<\delta\quad\Longrightarrow\quad \min_{f\in\mathcal O_{\bar\omega}}\left\|\zeta(t)-f\right\|_{L^p(\mathbb T)}<\varepsilon\quad \forall\,t\in\mathbb R.
\end{equation}
\end{remark}

 Theorem \ref{thm1} establishes the existence of a family of orbitally stable sinusoidal steady states on a flat 2-torus of arbitrary shape. For a rectangular or square torus (see Examples \ref{examp1} and \ref{examp2}), this result is already known. However, for a hexagonal torus (see Example \ref{examp3}), such steady states have not, to the best of our knowledge, appeared in the literature.

\begin{corollary}[Orbitally stable sinusoidal states on a hexagonal torus]\label{coro1}
Suppose that $\Lambda$ is given by \eqref{eisto1}.
Consider a steady state $\bar\omega\in\mathbb E_1,$ which can be written as
\[\bar\omega=A\cos\left(x_1-\frac{x_2}{\sqrt{3}}+\alpha\right)
+B\cos\left(\frac{2x_2}{\sqrt{3}}+\beta \right) +C\cos\left(x_1+\frac{x_2}{\sqrt{3}}+\gamma\right),\]
where $A,B,C\ge 0$   and $\alpha,\beta,\gamma\in\mathbb R.$
Then,  for any  $\varepsilon>0,$ there exists some $\delta>0,$ such that for any $L^p$-admissible map $\zeta(t)$ in the sense of Definition \ref{defam}, if
\[\|\zeta(0)-\bar\omega\|_{L^p(\mathbb T)}<\delta,\]
then for any $t\in\mathbb R,$ there exists some $\tilde\omega\in\mathbb E_1$ of the form
\[
\tilde\omega=A\cos\left(x_1-\frac{x_2}{\sqrt{3}}+\tilde\alpha\right)
+B\cos\left(\frac{2 x_2}{\sqrt{3}}+\tilde\beta \right) +C\cos\left(x_1+\frac{x_2}{\sqrt{3}}+\tilde\gamma\right),
\]
where $\tilde\alpha,\tilde\beta,\tilde\gamma\in\mathbb R$ satisfy \begin{equation}\label{pyix1}
ABCe^{{\rm i}(\alpha+\beta-\gamma)}
=ABCe^{{\rm i}(\tilde\alpha+\tilde\beta-\tilde\gamma)},
\end{equation}
 such that
\[ \|\zeta(t)-\tilde\omega\|_{L^p(\mathbb T)}<\varepsilon.\]
\end{corollary}

\begin{remark}
The  constraint  \eqref{pyix1} is to ensure that $\tilde\omega$ is a translation of $\bar\omega$ (see Lemma \ref{lemap2}(iii) in Appendix \ref{apdx2}). If $ABC=0,$ then \eqref{pyix1} is inactive; if $ABC\neq 0,$ then \eqref{pyix1} is equivalent to
\[\alpha + \beta - \gamma \equiv \tilde\alpha + \tilde\beta - \tilde\gamma \pmod{2\pi}.\]
\end{remark}

Recently, Jeong, Yao, and Zhou \cite{JYZ} showed on the standard square torus that, under an additional nondegeneracy assumption excluding smaller spatial periods, an orbitally stable steady state with a saddle point can be used to construct an $L^\infty$-open set of initial data exhibiting superlinear growth of the vorticity gradient. The states covered by Corollary \ref{coro1}, such as \eqref{sdpt}, provide natural candidates for extending this construction to a hexagonal torus; such an extension would require verifying the hypotheses and adapting the argument of \cite{JYZ} to this geometry.

\subsection{Comments and outline of the proof}\label{sec14}

To study the Lyapunov stability of a steady state of a two-dimensional ideal fluid, an effective approach is to use the conservation laws of the Euler equation to control the deviation of any perturbed solution from the steady state. The earliest use of this approach can be traced back to Arnold’s work \cite{A1,A2} in the 1960s, where he proposed the famous energy-Casimir (EC) functional method. For Laplacian eigenstates, the related Casimir is the enstrophy, i.e., the $L^2$-norm of the vorticity. By applying the conservations of the kinetic energy and the enstrophy, it can be shown that for any $\bar\omega\in\mathbb E_1$, the set
\[\mathcal S_{\bar\omega}:=\left\{f\in \mathbb E_1\mid \|f\|_{L^2(\mathbb T)}=\|\bar\omega\|_{L^2(\mathbb T)}\right\}\]
is  stable as in \eqref{sosv} with $p=2$.
Note that all the states in $\mathcal S_{\bar\omega}$ have the same kinetic energy and enstrophy. To distinguish between different states in $\mathcal{S}_{\bar\omega}$ for a square torus, Wirosoetisno and Shepherd \cite{WS99} presented an analysis involving higher-order (cubic, quartic, and quintic) Casimirs. However, their formulation of orbital stability depends on higher-order Casimirs, and complete orbital stability therefore remains unclear. The first complete orbital stability result, measured in the $L^p$-norm of the vorticity for any $1 < p < \infty$, was proved in \cite{WZCV} for both the rectangular and the square torus.
In contrast to the approach of Wirosoetisno and Shepherd \cite{WS99}, the proof in \cite{WZCV} was achieved within the framework of Burton’s stability theory, with a key ingredient being the analysis of the equimeasurable partition of the first eigenspace.
Subsequently,   Elgindi \cite{ElCPAA} obtained a quantitative $L^2$-stability result by improving Wirosoetisno and Shepherd's argument for the square torus.

Our approach to proving Theorem \ref{thm1} is primarily inspired by \cite{WZCV} and can be outlined in the following three steps:
\begin{itemize}
  \item[(1)] \emph{Variational characterization for the equimeasurable class $\mathcal C_{\bar\omega}$ of $\bar\omega$ in $\mathbb E_1$.} The equimeasurable class $\mathcal C_{\bar\omega}$ of $\bar\omega$ in $\mathbb E_1$ is defined as the set of all functions in $\mathbb E_1$ that are equimeasurable with $\bar\omega$, or equivalently,
  \begin{equation}\label{cobt}
  \mathcal C_{\bar\omega} = \mathcal R_{\bar\omega} \cap \mathbb E_1.
  \end{equation}
   By applying the energy-enstrophy inequality, we show that $\mathcal C_{\bar\omega}$ can be characterized via the conserved quantities of the Euler equation; more precisely, $\mathcal C_{\bar\omega}$ is exactly the set of maximizers of the kinetic energy $E$ relative to the rearrangement class $\mathcal R_{\bar\omega}$. The step is carried out in Section \ref{sec31}.

  \item[(2)] \emph{Isolatedness of the translational orbit $\mathcal O_{\bar\omega}$ in $\mathcal C_{\bar\omega}$.} This is the most challenging step, and is carried out in Section \ref{sec32}.
The main difficulty arises in the hexagonal case.
There the first eigenspace is six-dimensional and its three positive wave
vectors satisfy the resonance relation
$\bm k_3=\bm k_1+\bm k_2$. Three Fourier phases are therefore acted on by
only two translation parameters, leaving
$\alpha_1+\alpha_2-\alpha_3$ invariant. Using the Casimirs of orders $2,3,4$, and $6$, we derive the polynomial system \eqref{ase00}; Lemma \ref{lemap3} shows that this system has only finitely many solutions.
This is the main technical issue, which does not arise for the rectangular or square tori considered in \cite{WZCV}.
The lattice classification in Lemma \ref{lemdime1} and the
translation criterion in Lemma \ref{lemap2} make this reduction precise.
 \item[(3)] \emph{Application of a Burton-type stability criterion.}   In the spirit of Burton \cite{BAR}, it can be shown that the set of maximizers of $E$ relative to $\mathcal R_{\bar\omega}$ is stable under the Euler dynamics (see Proposition \ref{propbsc} in Section \ref{sec23}). Combining the previous two steps with a continuity argument, we conclude that the translational orbit $\mathcal O_{\bar\omega}$ is also stable under the Euler dynamics.
\end{itemize}

Note that the above approach is also effective for addressing the stability of Laplacian eigenstates in other symmetric domains, such as a disk \cite{Wdisk}, a rotating sphere \cite{CWZ}, and a finite periodic channel \cite{Wchanl}. In addition, this approach can yield $L^p$-stability, which seems difficult to achieve using the methods in \cite{ElCPAA,WS99}.

Finally, we emphasize that the stability of an individual representative, without quotienting out translations, remains an open problem.
Indeed, it is impossible to distinguish between two states in $\mathcal O_{\bar\omega}$ using only the conservation laws (C1) and (C2).
See also \cite{ElCPAA} for a more detailed discussion of this issue.

This paper is organized as follows. In Section~\ref{sec2}, we present some preliminary results that will be used throughout the paper. In Section~\ref{sec3}, we prove Theorem~\ref{thm1}. Section~\ref{sec4} establishes an interesting rigidity result that characterizes the borderline case of Arnold-stable states. For clarity, the proofs of several lemmas are deferred to the three appendices.

\section{Preliminaries}\label{sec2}

\subsection{Laplacian eigenvalue problem on $\mathbb T$}\label{sec21}

The results in this subsection may be familiar to experts, but we provide detailed proofs for the reader’s convenience.

\begin{lemma}\label{lemeigen1}
The set of  eigenvalues
 for \eqref{evp} is given by \eqref{evevp},
and the eigenspace related to some eigenvalue $\lambda$  is given by \eqref{efevp}.
\end{lemma}
\begin{proof}
It is clear that $e^{2\pi {\rm i}\bm k\cdot \bm x}$   is  an eigenfunction of \eqref{evp} for any $\bm k\in\Lambda^*\setminus\{(0,0)\}$ with  $4\pi^2|\bm k|^2$ being the associated eigenvalue.   Notice that \eqref{evp} is equivalent to the following operator equation:
\[v=\lambda\mathsf G v,\quad v\in\mathring L^2(\mathbb T),\]
where $\mathsf G:\mathring L^2(\mathbb T)\to \mathring L^2(\mathbb T)$ is compact and symmetric (cf. Appendix \ref{apdx1}).
According to the Hilbert-Schmidt theory, it suffices to show that $\left\{e^{2\pi {\rm i}\bm k\cdot \bm x}\mid \bm k\in\Lambda^*\setminus\{(0,0)\} \right\}$ is complete in $\mathring L^2(\mathbb T)$; or equivalently, for any $f\in  \mathring L^2(\mathbb T)$ satisfying
\begin{equation}\label{equ0}
\int_{\mathbb T}f(\bm x)e^{2\pi {\rm i}\bm k\cdot \bm x}d\bm x=0\quad\forall\,\bm k\in\Lambda^*\setminus\{(0,0)\},
\end{equation}
it holds that $f=0$ a.e. on $\mathbb T.$ Note that \eqref{equ0} can be written as
\begin{equation}\label{equ01}
\int_{\mathbb T}f(\bm x)e^{2\pi {\rm i}(m\bm\xi^*+n\bm\eta^*)\cdot \bm x}d\bm x=0\quad\forall\,(m,n)\in\mathbb Z^2\setminus\{(0,0)\},
\end{equation}
where $\bm\xi^*$ and $\bm\eta^*$ are given by \eqref{dulbas}.
By the change of variables
\[\bm x=y_1\bm\xi+y_2\bm\eta,\quad 0<y_1, y_2<1, \]
and using \eqref{dulbas2},
\eqref{equ01} becomes
\[
\int_0^1\int_0^1f(y_1\bm\xi+y_2\bm\eta)e^{2\pi {\rm i}\bm k\cdot\bm y }dy_1 dy_2=0\quad\forall\, \bm k\in\mathbb Z^2\setminus\{(0,0)\}.
\]
Since $\left\{e^{2\pi {\rm i} \bm k\cdot\bm x}\mid \bm k\in\mathbb Z^2\setminus\{(0,0)\}\right\}$ forms an orthonormal basis of $\mathring L^2((0,1)\times(0,1))$ (see \cite[p. 32]{Ti}), we further deduce that  $f(y_1\bm\xi+y_2\bm\eta)=0$ for a.e. $(y_1,y_2)\in(0,1)\times(0,1)$, and hence $f=0$ a.e. on $\mathbb T$. This completes the proof.
\end{proof}

\begin{lemma}[Dimension of $\mathbb E_1$]\label{lemdime1}
The dimension of $\mathbb{E}_1$ is either 2, 4, or 6. Moreover,
\begin{itemize}
\item[(i)] If $\dim(\mathbb{E}_1)=2$, then there exists some nonzero vector $\bm k$ such that
    \[S(\Lambda^*)=\{\pm\bm k\}.\]
     Accordingly,
    \[\mathbb E_1={\rm span}\left\{\cos(2\pi\bm k\cdot\bm x), \sin(2\pi\bm k\cdot\bm x)\right\}.\]
\item[(ii)] If $\dim(\mathbb{E}_1)=4$, then there exist  two linearly independent vectors $\bm k_1, \bm k_2$ satisfying $|\bm k_1|=|\bm k_2|$ such that \[S(\Lambda^*)=\{\pm\bm k_1, \pm\bm k_2\}.\]
    Accordingly,
    \[\mathbb E_1 ={\rm span}\left\{\cos(2\pi \bm k_1\cdot\bm x), \sin(2\pi \bm k_1\cdot\bm x), \cos(2\pi \bm k_2\cdot\bm x), \sin(2\pi \bm k_2\cdot\bm x)\right\}.\]
\item[(iii)] If $\dim(\mathbb E_1)=6$, then $\Lambda^*$ is a hexagonal lattice, and there exist three nonzero vectors  $\bm k_1,\bm k_2, \bm k_3$ satisfying $|\bm k_1|=|\bm k_2|=|\bm k_3|$ and $\bm k_3=\bm k_1+\bm k_2$ such that
\[S(\Lambda^*)=\left\{ \pm\bm k_1,\pm\bm k_2, \pm\bm k_3 \right\}.\]
Moreover, $(\bm k_1,\bm k_2)$ is a basis of $\Lambda^*$.
Accordingly, $\mathbb E_1$ is spanned by the following six functions:
\[ \cos(2\pi \bm k_1\cdot\bm x), \, \sin(2\pi \bm k_1\cdot\bm x), \,\cos(2\pi \bm k_2\cdot\bm x),\, \sin(2\pi \bm k_2\cdot\bm x), \,\cos(2\pi \bm k_3\cdot\bm x),\, \sin(2\pi \bm k_3\cdot\bm x).\]
    \end{itemize}
  \end{lemma}
\begin{proof}
By \eqref{numbs}, $\dim\mathbb E_1=\#S(\Lambda^*)$. Since the set
$S(\Lambda^*)$ is centrally symmetric (i.e., $\bm k\in S(\Lambda^*)$ if and only if $-\bm k\in S(\Lambda^*)$), its cardinality is even. If
$\bm k,\bm l\in S(\Lambda^*)$ are distinct, then
$\bm k-\bm l\in\Lambda^*\setminus\{0\}$ and hence
\[
 |\bm k-\bm l|\geq\rho(\Lambda^*)=|\bm k|=|\bm l|.
\]
Thus the angle between consecutive shortest vectors is at least $\pi/3$,
and there are at most six of them. This proves the alternatives in (i)-(iii),
except for the final basis assertion.

In the six-vector case the angular gaps are all $\pi/3$. Choose
$\bm k_1$ and $\bm k_2$ with angle $2\pi/3$ and set
$\bm k_3=\bm k_1+\bm k_2$. Then
$\bm k_1\cdot\bm k_2=-\rho(\Lambda^*)^2/2$, so $\bm k_3$ is another
shortest vector and the displayed description of $S(\Lambda^*)$ follows.
Let $L_0=\mathbb Z\bm k_1+\mathbb Z\bm k_2$. If
$L_0\ne\Lambda^*$, choose $\bm q_0\in\Lambda^*\setminus L_0$ and reduce it
modulo $L_0$ to a nonzero vector $\bm q=a\bm k_1+b\bm k_2$ with
$|a|,|b|\leq1/2$. But then
\[
 0<|\bm q|^2
 =\rho(\Lambda^*)^2(a^2+b^2-ab)
 \leq\frac34\rho(\Lambda^*)^2,
\]
contradicting the definition of $\rho(\Lambda^*)$. Hence
$L_0=\Lambda^*$. Since $\bm k_3=\bm k_1+\bm k_2$, the pair
$(\bm k_1,\bm k_3)$ is also a basis of $\Lambda^*$, and
$\bm k_1\cdot\bm k_3=\rho(\Lambda^*)^2/2$. Thus this basis has equal
lengths and angle $\pi/3$, proving that $\Lambda^*$ is hexagonal.

\end{proof}

\begin{remark}
  If $\dim(\mathbb E_1)=6$, then $\mathbb T$ must be a hexagonal  torus. Indeed, if $\dim(\mathbb E_1)=6$, then $\Lambda^{*}$ is a hexagonal  lattice by Lemma \ref{lemdime1}(iii), and thus $\Lambda^{**}$ is also a hexagonal  lattice. Since
$\Lambda = \Lambda^{**}$ (which follows directly from the definition of the dual lattice), we see that $\Lambda$ is also a hexagonal  lattice.
\end{remark}

\subsection{Energy-enstrophy inequality}\label{sec22}

Recall the following Poincar\'e inequality, which can be proved via an eigenfunction expansion or a standard variational argument.
\begin{lemma}[Poincar\'e inequality]\label{lempoincare}
For any $u\in  \mathring H^1(\mathbb T)$, it holds that
 \begin{equation}\label{pocar1}
 \lambda_1\int_{\mathbb T} u^2d\bm x\leq \int_{\mathbb T}|\nabla u|^2d\bm x,
 \end{equation}
and equality holds if and only if $u\in\mathbb E_1.$
\end{lemma}

The following energy-enstrophy inequality, which is a direct corollary of the Poincar\'e inequality, will play an important role in the proof of Proposition \ref{propmc} in Section \ref{sec31}.

\begin{lemma}[Energy-enstrophy inequality]\label{lemeei}
For any $f\in \mathring L^2(\mathbb T)$, it holds that
\[\int_{\mathbb T} f\mathsf Gf d\bm x\leq \frac{1}{\lambda_1}\int_{\mathbb T} f^2d\bm x,\]
and equality holds if and only if $f\in\mathbb E_1.$
\end{lemma}
\begin{proof}
We estimate as follows:
  \begin{equation}\label{pocar2}
  \begin{split}
  2\int_{\mathbb T} f\mathsf Gf d\bm x&\leq \frac{1}{\lambda_1}\int_{\mathbb T}f^2 d\bm x+ \lambda_1\int_{\mathbb T}|\mathsf Gf|^2 d\bm x\\
  & \leq \frac{1}{ \lambda_1}\int_{\mathbb T}f^2 d\bm x+  \int_{\mathbb T}|\nabla\mathsf Gf|^2 d\bm x\\
   & = \frac{1}{ \lambda_1}\int_{\mathbb T}f^2 d\bm x+  \int_{\mathbb T}f\mathsf Gf d\bm x,
   \end{split}
  \end{equation}
where we have used the inequality of arithmetic and geometric means in the first inequality, and \eqref{pocar1} in the second inequality.
Moreover, the first inequality in \eqref{pocar2} is an equality if and only if $f=\lambda_1 \mathsf Gf$, which is equivalent to $f\in\mathbb E_1;$ and the second one is an equality if and only if $\mathsf Gf\in\mathbb E_1$, which is also equivalent to $f\in\mathbb E_1.$  This completes the proof.

\end{proof}

\begin{remark}
 Lemma \ref{lemeei} provides a variational characterization for $\mathbb E_1$ in terms of the kinetic energy and the enstrophy (both conserved quantities of the Euler equation), i.e., $\mathbb E_1$ is exactly the set of maximizers of the following maximization problem:
 \[\sup_{f\in \mathring L^2(\mathbb T),\,f\not\equiv 0}\frac{\int_{\mathbb T}f\mathsf G f d\bm x}{\int_{\mathbb T}f^2d\bm x}.\]
  Such a variational characterization is the very basis on which we can analyze the stability of the first eigenstates.
\end{remark}

\subsection{A Burton-type stability criterion}\label{sec23}
We now state a stability criterion for two-dimensional ideal fluids related to the maximization of the kinetic energy, sometimes together with certain linear conserved quantities associated with the symmetry of the domain, relative to a fixed rearrangement class. This idea originates from Burton’s work \cite{BAR}, with later developments discussed in \cite{CWZ, WMA, Wdisk, Wchanl, WZCV}.

Let $\mathcal R$ be the rearrangement class of some function in
$\mathring L^p(\mathbb T)$. Consider the following maximization problem:
\begin{equation}\label{mmpm}
M=\sup_{f\in \mathcal R}E(f).
\end{equation}
Denote by $\mathcal M$ the set  of maximizers for \eqref{mmpm}, i.e.,
\[\mathcal M=\left\{f\in \mathcal R\mid E(f)=M\right\}.\]

\begin{proposition}\label{propbsc}
The set $\mathcal M$ is  nonempty and compact  in $\mathring L^p(\mathbb T)$, and is stable  in the following sense: for any $\varepsilon>0,$ there exists some $\delta>0, $ such that for any $L^p$-admissible map $\zeta(t)$ in the sense of Definition \ref{defam}, if
\[\min_{f\in\mathcal M}\|\zeta(0)-f\|_{L^p(\mathbb T)}<\delta\]
then
\[ \min_{f\in\mathcal M}\|\zeta(t)-f\|_{L^p(\mathbb T)}<\varepsilon\quad\forall\,t\in\mathbb R.\]
 \end{proposition}
 \begin{proof}
  It follows from a similar argument as  in \cite[Section 5]{Wdisk}. For the reader's convenience, we provide a detailed proof below.
   
We first prove that $\mathcal M$ is nonempty and compact by establishing the compactness of maximizing sequences for \eqref{mmpm}. Let $\{f_n\}\subset\mathcal R$ satisfy $E(f_n)\to M$. Since $\mathcal R$ is bounded in $L^p(\mathbb T)$ and $1<p<\infty$, after passing to a subsequence we have
\[
f_n\rightharpoonup f\quad\hbox{weakly in }L^p(\mathbb T)
\]
for some $f\in \mathring L^p(\mathbb T)$. Let $p'=p/(p-1)$. Since the map $\mathsf G:\mathring L^p(\mathbb T)\to \mathring L^{p'}(\mathbb T)$ is compact (see Appendix \ref{apdx1}), $E$ is weakly sequentially continuous in $\mathring L^p(\mathbb T)$. Hence $E(f)=M$.
Applying \cite[Theorem 4]{BMA}, there exists $g\in\mathcal R$ such that
\[
\int_{\mathbb T}g\mathsf Gf\,d\bm x
=\sup_{v\in\mathcal R}\int_{\mathbb T}v\mathsf Gf\,d\bm x.
\]
It follows that
\[
\int_{\mathbb T}g\mathsf Gf\,d\bm x\geq \lim_{n\to\infty}\int_{\mathbb T}f_n\mathsf Gf\,d\bm x
= \int_{\mathbb T}f\mathsf Gf\,d\bm x=2M.
\]
Using the symmetry and positive definiteness of $\mathsf G$, we obtain
\[
0\leq E(g-f)
=E(g)+E(f)-\int_{\mathbb T}g\mathsf Gf\,d\bm x
\leq M+M-2M=0.
\]
Thus $g=f$, and hence $f\in\mathcal M$. Moreover,
$\|f_n\|_{L^p(\mathbb T)}=\|f\|_{L^p(\mathbb T)}$ since $f_n,f\in\mathcal R$. Then the uniform convexity of $L^p(\mathbb T)$, together with the weak convergence $f_n\rightharpoonup f$, yields 
\[
f_n\to f
\quad\hbox{strongly in }L^p(\mathbb T).
\]
Consequently, every maximizing sequence has a strongly convergent subsequence with limit in $\mathcal M$. This proves that $\mathcal M$ is nonempty and compact.

It remains to prove stability. Suppose, to the contrary, that the conclusion fails. Then there exist $\varepsilon_0>0$, a sequence of $L^p$-admissible maps $\{\zeta_n\}$, and a sequence of times $\{t_n\}\subset\mathbb R$ such that
\begin{equation}\label{cont1}
\min_{f\in\mathcal M}\|\zeta_n(0)-f\|_{L^p(\mathbb T)}\to 0,
\end{equation}
and
\begin{equation}\label{cont2}
\min_{f\in\mathcal M}\|\zeta_n(t_n)-f\|_{L^p(\mathbb T)}\geq\varepsilon_0.
\end{equation}
In view of \eqref{cont1} and the compactness of $\mathcal M$, after taking a subsequence we have that $\zeta_n(0)\to f$ in $L^p(\mathbb T)$ for some $f\in\mathcal M$. Since $\zeta_n(t_n)\in\mathcal R_{\zeta_n(0)}$ and $f\in\mathcal R$, we can apply \cite[Lemma 2.3]{BACT} to find some $\eta_n\in\mathcal R$ such that  
\[
\|\zeta_n(t_n)-\eta_n\|_{L^p(\mathbb T)}
\leq \|\zeta_n(0)-f\|_{L^p(\mathbb T)}\to 0.
\]
It is clear that the sequence $\{\eta_n\}$ is bounded in $L^p(\mathbb T)$; by the preceding convergence, so is $\{\zeta_n(t_n)\}$. On the other hand, the boundedness of $\mathsf G:\mathring L^p(\mathbb T)\to \mathring L^{p'}(\mathbb T)$ implies that
\[
|E(\zeta_n(t_n))-E(\eta_n)|\leq C\left(\|\zeta_n(t_n)\|_{L^p(\mathbb T)}+\|\eta_n\|_{L^p(\mathbb T)}\right)\|\zeta_n(t_n)-\eta_n\|_{L^p(\mathbb T)},
\]
where $C>0$ depends only on $p$ and $\Lambda$.
Consequently, $E(\eta_n)-E(\zeta_n(t_n))\to0$. Energy conservation and the continuity of $E$ at $f$ now give
\[
E(\eta_n)-M
=E(\eta_n)-E(\zeta_n(t_n))+E(\zeta_n(0))-E(f)\to0.
\]
Thus $\{\eta_n\}$ is a maximizing sequence for \eqref{mmpm}; by the compactness just proved, a subsequence converges strongly to an element of $\mathcal M$. The same is then true of $\zeta_n(t_n)$, contradicting \eqref{cont2}. This completes the proof.
 \end{proof}

\begin{remark} 
The preceding proof uses reflexivity and uniform convexity of $L^p(\mathbb T)$, as well
as the compact mapping
$\mathsf G:\mathring L^p(\mathbb T)\to\mathring L^{p'}(\mathbb T)$. These ingredients do not give
the required strong compactness at $p=1$ or $p=\infty$. The restriction
$1<p<\infty$ concerns this method and does not assert instability at either endpoint.
\end{remark}

\section{Proof}\label{sec3}

Throughout this section, let $\bar\omega\in\mathbb E_1$ be fixed. Recall that $\mathcal O_{\bar\omega}$ and $\mathcal C_{\bar\omega}$ are defined in \eqref{tobt} and \eqref{cobt}, respectively.

\subsection{Variational characterization for $\mathcal C_{\bar\omega}$}\label{sec31}
 The aim of this subsection is to prove the following proposition.

 \begin{proposition}\label{propmc}
 Consider the following maximization problem:
\begin{equation}\label{mmpm2}
M_{\bar\omega}=\sup_{f\in \mathcal R_{\bar\omega}}E(f).
\end{equation}
Denote by $\mathcal M_{\bar\omega}$ the set  of maximizers for \eqref{mmpm2}. Then
 \[
 \mathcal M_{\bar\omega}=\mathcal C_{\bar\omega}.
 \]
  \end{proposition}
  \begin{proof}
 Since $\mathcal C_{\bar\omega}=\mathcal R_{\bar\omega}\cap \mathbb E_1,$ it suffices to prove the following claim:
   \begin{equation}\label{ccle0}
   E(\bar\omega)\geq E(f)\mbox{ for any }f\in\mathcal R_{\bar\omega},\mbox{ and equality holds if and only if }f\in \mathbb E_1.
   \end{equation}
To this end,  fix an arbitrary $f\in\mathcal R_{\bar\omega}$. For convenience, write $f= \bar\omega+\varrho$. Then  $\|\bar\omega+\varrho\|_{L^2(\mathbb T)}=\|\bar\omega\|_{L^2(\mathbb T)}$, which implies
 \begin{equation}\label{ccle2}
 \int_{\mathbb T}\varrho\bar\omega d\bm x=-\frac{1}{2}\int_{\mathbb T}\varrho^2 d\bm x.
 \end{equation}
Using \eqref{ccle2}, we can compute as follows:
    \begin{equation}\label{ccle3}
    \begin{split}
    E(\bar\omega)-E(\varrho+\bar\omega)&=-\frac{1}{2}\int_{\mathbb T}\varrho\mathsf G\varrho d\bm x-\int_{\mathbb T}\varrho\mathsf G\bar\omega d\bm x\\
    &=-\frac{1}{2}\int_{\mathbb T}\varrho\mathsf G\varrho d\bm x-\frac{1}{\lambda_1}\int_{\mathbb T}\varrho \bar\omega d\bm x\\
    &=-\frac{1}{2}\int_{\mathbb T}\varrho\mathsf G\varrho d\bm x+\frac{1}{2\lambda_1}\int_{\mathbb T}\varrho^2d\bm x.
    \end{split}
    \end{equation}
Note that we used $\mathsf G\bar\omega=\lambda_1^{-1}\bar\omega$ (since $\bar\omega\in\mathbb E_1$) in the second equality of \eqref{ccle3}. Then, by applying the energy–enstrophy inequality (see Lemma \ref{lemeei}) to $\varrho$, we deduce that
   \[ E(\bar\omega)-E(\varrho+\bar\omega)\geq 0,\]
with equality if and only if $\varrho\in\mathbb E_1$, which is equivalent to $f\in\mathbb E_1$. The proof is complete.
  \end{proof}

 \subsection{Finiteness of translational orbits within $\mathcal C_{\bar\omega}$}\label{sec32}


In this section, we show that the number of translational orbits within $\mathcal C_{\bar\omega}$ is finite. For clarity, we divide the discussion into three cases according to the dimension of $\mathbb E_1.$

\subsubsection{Case $\dim(\mathbb E_1)=2$}\label{sec321}
 \begin{proposition}\label{prop2d}
If $\dim(\mathbb E_1)=2$, then $\mathcal C_{\bar\omega} =\mathcal O_{\bar\omega}$ (i.e., there is only one translational orbit in $\mathcal C_{\bar\omega}$).
 \end{proposition}

 \begin{proof}
 Since the inclusion $\mathcal O_{\bar\omega} \subset \mathcal C_{\bar\omega}$ is obvious, it suffices to prove
 \begin{equation}\label{csso1}
 \mathcal C_{\bar\omega} \subset \mathcal O_{\bar\omega}.
 \end{equation}
To this end, fix $w\in\mathcal C_{\bar\omega}$. Let $\bm k$ be as in Lemma \ref{lemdime1}(i). Then there exist $A,B\geq 0$ and $\alpha,\beta\in\mathbb R$ such that
\[\bar\omega= A\cos(2\pi\bm k\cdot\bm x+\alpha), \quad w= B\cos(2\pi\bm k\cdot\bm x+\beta).\]
Since $w\in\mathcal C_{\bar\omega}\subset\mathcal R_{\bar\omega}$, it holds that
\[\|\bar\omega\|_{L^\infty(\mathbb T)}=\|w\|_{L^\infty(\mathbb T)}.\]
On the other hand, it is obvious that
\[ \|\bar\omega\|_{L^\infty(\mathbb T)}=A,\quad \|w\|_{L^\infty(\mathbb T)}=B.\]
Hence  $A= B,$ which implies that $w\in\mathcal O_{\bar\omega}$ by Lemma \ref{lemap2}(i). The proof is complete.
 \end{proof}

\subsubsection{Case $\dim(\mathbb E_1)=4$}\label{sec322}
 \begin{proposition}\label{prop4d}
If $\dim(\mathbb E_1)=4$, then there are at most 2 translational orbits within $\mathcal C_{\bar\omega}$.
 \end{proposition}

  \begin{proof}
 Fix an arbitrary translational orbit $\mathcal O_{w}$ with $w\in\mathcal C_{\bar\omega}$. Let $\bm k_1$ and $\bm k_2$ be as in Lemma \ref{lemdime1}(ii). Then there exist $A_1, A_2\geq 0$ and $\alpha_1, \alpha_2\in\mathbb R$ such that
  \[w= \sum_{i=1}^2A_i \cos(2\pi\bm k_i\cdot\bm x+\alpha_i).\]
 Since $w\in \mathcal C_{\bar\omega}\subset \mathcal R_{\bar\omega},$ it holds that  \[\|w\|_{L^\infty(\mathbb T)}=\|\bar\omega\|_{L^\infty(\mathbb T)},\quad \|w\|_{L^2(\mathbb T)}=\|\bar\omega\|_{L^2(\mathbb T)},\]
  which yields
 \begin{equation}\label{a1a2p1}
 A_1+A_2=a_1,\quad
  p_1A_1^2+p_2A_2^2=a_2,
  \end{equation}
  where
    \[p_1=\int_{\mathbb T}\cos^2(2\pi\bm k_1\cdot\bm x)d\bm x,\quad p_2=\int_{\mathbb T}\cos^2(2\pi\bm k_2\cdot\bm x)d\bm x,\]
  and $a_1, a_2\in\mathbb R$ depend only on $\bar\omega$ and $\Lambda$. Note that, in deriving the first relation in \eqref{a1a2p1}, we used the fact that $\bm k_1$ and $\bm k_2$ are linearly independent.
The desired result is then a straightforward consequence of the following two facts:
\begin{itemize}
  \item [(1)]  There are at most two pairs $(A_1,A_2)$ satisfying \eqref{a1a2p1}.
  \item  [(2)] The translational orbit $\mathcal O_{w}$ is uniquely determined by the pair $(A_1,A_2)$; see  Lemma \ref{lemap2}(ii) in Appendix \ref{apdx2}.

\end{itemize}

  \end{proof}

\subsubsection{Case $\dim(\mathbb E_1)=6$}\label{sec323}

 \begin{proposition}\label{prop6d}
If $\dim(\mathbb E_1)=6$, then there are at most 12 translational orbits within $\mathcal C_{\bar\omega}$.
 \end{proposition}

  \begin{proof}

The proof follows a similar argument to that in the 4D case.
Fix an arbitrary translational orbit $\mathcal O_{w}$ with $w\in\mathcal C_{\bar\omega}$.
Let $\bm k_1, \bm k_2$ and $\bm k_3$ be as in Lemma \ref{lemdime1}(iii). By Lemma \ref{lemdime1}(iii) and the fact that $\bm k_1, \bm k_2$ are linearly independent, we may assume, without loss of generality, that $w$ has the following form:
\[w= A_1\cos(2\pi\bm k_1\cdot\bm x)+A_2\cos(2\pi\bm k_2\cdot\bm x)+A_3\cos(2\pi\bm k_3\cdot\bm x+\alpha),\]
where $A_i\geq 0$ for $i=1,2,3$, and $\alpha\in\mathbb R$.  Since $w\in\mathcal C_{\bar\omega}\subset\mathcal R_{\bar\omega}$, we have
\begin{equation}\label{intm1}
\int_{\mathbb T}w^m d\bm x=\int_{\mathbb T}\bar\omega^m d\bm x
\end{equation}
for any positive integer $m$.
Since $\Lambda^{**}=\Lambda$ and $(\bm k_1,\bm k_2)$ is a basis of $\Lambda^*,$ we can choose $(\bm k_1^*,\bm k_2^*)$, defined as in \eqref{dulbas}, as a new basis of $\Lambda.$ Then, as in \eqref{dulbas2}, it holds that
 \[
\bm k_1^*\cdot\bm k_1=1,\quad \bm k_1^*\cdot\bm k_2=\bm k_1\cdot\bm k_2^*=0,\quad \bm k_2^*\cdot\bm k_2=1.
\]
 By the change of variables
 \[\bm x=\frac{1}{2\pi}\left(y_1\bm k_1^*+y_2\bm k_2^*\right),\quad 0<y_1,y_2<2\pi,\]
and using  $\bm k_3=\bm k_1+\bm k_2$,    \eqref{intm1}  becomes
\begin{equation}\label{02pi}
\begin{split}
&\int_{0}^{2\pi}\int_{0}^{2\pi}(A_1\cos y_1+A_2\cos  y_2 +A_3\cos(y_1+y_2+\alpha))^mdy_1dy_2=a_m,
\end{split}
\end{equation}
where $a_m\in\mathbb R$ depends only on $\bar\omega$, $m$ and $\Lambda$.
By taking $m=2,3,4,6$ in \eqref{02pi}, respectively, we obtain
\begin{equation}\label{meq}
\begin{cases}
A_1^2+A_2^2+A_3^2=b_1,\\
A_1 A_2 A_3\cos\alpha =b_2,\\
A_1^4+A_2^4+A_3^4+4\left(A_1^2A_2^2+A_1^2A_3^2+A_2^2A_3^2\right)=b_3,\\
A_1^6+A_2^6+A_3^6+9\left(A_1^4A_2^2+A_1^4A_3^2+A_1^2A_2^4+A_2^4A_3^2+A_1^2A_3^4+A_2^2A_3^4\right)\\
\quad\ +27A_1^2A_2^2A_3^2   +18A_1^2A_2^2A_3^2\cos^2\alpha
=b_4,
\end{cases}
\end{equation}
where $b_1, b_2, b_3, b_4$ depend only on  $\bar\omega$ and $\Lambda$. 
From \eqref{meq}, we see that $(A_1^2,A_2^2,A_3^2)$  is a solution to the following system of polynomial equations:
\begin{equation}\label{ase00}
\begin{cases}
 x+ y+ z=c_1,\\
 x^2+ y^2+ z^2+4\left(x y+xz+yz\right)=c_2,\\
  x^3+ y^3+ z^3+9\left(x^2y+x^2z+xy^2+y^2z+xz^2+yz^2\right)+27 x y z=c_3,
 \end{cases}
\end{equation}
where $c_1, c_2, c_3$  depend only on  $\bar\omega$ and $\Lambda$.
By Lemma \ref{lemap3} in Appendix \ref{apdx3}, the system \eqref{ase00} has at most 6 solutions. In other words, there are at most 6 triples $(A_1,A_2,A_3)$ satisfying \eqref{meq}.

To conclude the proof, it suffices to show that each triple $(A_1,A_2,A_3)$ satisfying \eqref{meq} determines at most 2 translational orbits.
We distinguish two cases:
 \begin{itemize}
 \item[(1)] $A_1A_2A_3=0$. In this case,   each triple $(A_1,A_2,A_3)$  determines  a single translational orbit by Lemma \ref{lemap2}(iii) in Appendix \ref{apdx2}.
\item[(2)]  $A_1A_2A_3\neq 0.$ In this case,   by Lemma \ref{lemap2}(iii) again, $\mathcal O_w$ is uniquely determined by the 5-tuple $(A_1,A_2,A_3,\cos\alpha, \sin\alpha)$. Observe that when $A_1 A_2A_3 \neq 0$, $\cos \alpha$, and hence $|\sin \alpha|$,  are uniquely determined by \eqref{meq}$_2$. Thus, given a triple $(A_1, A_2, A_3)$ satisfying \eqref{meq}, there are at most two 5-tuples $(A_1, A_2, A_3, \cos\alpha, \sin\alpha)$ satisfying \eqref{meq}, which implies that each triple $(A_1, A_2, A_3)$  determines at most two translational orbits.
\end{itemize}
The proof is complete.

  \end{proof}

\subsection{Proof of Theorem \ref{thm1}}\label{sec33}

\begin{lemma}\label{lemisolated}
Every translational orbit $\mathcal O\subset \mathcal C_{\bar\omega}$ is isolated in $\mathcal C_{\bar\omega}$, i.e.,
either $\mathcal O=\mathcal C_{\bar\omega}$, or
\[\mathcal O\neq \mathcal C_{\bar\omega}\quad\mbox{and}\quad\min\left\{\|f-g\|_{L^p(\mathbb T)} \mid f\in\mathcal O,\,g\in \mathcal C_{\bar\omega}\setminus \mathcal O\right\}>0.\]
\end{lemma}

\begin{proof}
From Propositions \ref{prop2d}, \ref{prop4d}, and \ref{prop6d}, we know that $\mathcal{C}_{\bar\omega}$ is the union of finitely many pairwise disjoint translational orbits. It is clear that each orbit is compact in $L^p(\mathbb T)$. Hence, if $\mathcal O\neq\mathcal C_{\bar\omega}$, then $\mathcal C_{\bar\omega}\setminus\mathcal O$ is a finite union of compact sets and is therefore compact. Two disjoint compact subsets of a metric space have positive distance, which proves the claim.
\end{proof}

Now we are ready to prove Theorem \ref{thm1} via a continuity argument.

\begin{proof}[Proof of Theorem \ref{thm1}]

From Propositions \ref{propbsc} and \ref{propmc}, we know that $\mathcal C_{\bar\omega}$ is stable in the sense of Proposition \ref{propbsc}. If $\mathcal O_{\bar\omega}=\mathcal C_{\bar\omega}$, the conclusion follows immediately. Otherwise, Lemma \ref{lemisolated} gives
\[
d:=\min\left\{\|f-g\|_{L^p(\mathbb T)} \mid f\in\mathcal O_{\bar\omega},\,g\in \mathcal C_{\bar\omega}\setminus \mathcal O_{\bar\omega}\right\} >0.
\]
Fix $\varepsilon>0$ and choose $0<r<\min\{\varepsilon,d/3\}.$
By the stability of $\mathcal C_{\bar\omega}$, there exists $\delta_0>0$ such that
\[
\min_{f\in \mathcal C_{\bar\omega}}\|\zeta(0)-f\|_{L^p(\mathbb T)}<\delta_0
\quad\Longrightarrow\quad
\min_{f\in \mathcal C_{\bar\omega}}\|\zeta(t)-f\|_{L^p(\mathbb T)} <r
\quad\forall\,t\in\mathbb R.
\]
Set $\delta=\min\{\delta_0,r\}$. If
$\|\zeta(0)-\bar\omega\|_{L^p}<\delta$, then $\zeta(0)$ belongs to the $r$-neighborhood of $\mathcal O_{\bar\omega}$. For every $t$, the point $\zeta(t)$ belongs to the union of the $r$-neighborhoods of $\mathcal O_{\bar\omega}$ and $\mathcal C_{\bar\omega}\setminus\mathcal O_{\bar\omega}$. These two neighborhoods are disjoint since $2r<d$. On the other hand, $t\mapsto\zeta(t)$ is continuous and $\mathbb R$ is connected, so the trajectory cannot pass from one neighborhood to the other. Therefore,
\[
\operatorname{dist}_{L^p}(\zeta(t),\mathcal O_{\bar\omega})<r<\varepsilon
\quad\forall\,t\in\mathbb R,
\]
which is the desired conclusion.

\end{proof}

\section{A rigidity result}\label{sec4}

Suppose that $u\in C^2(\mathbb T)$ satisfies
\begin{equation}\label{selp}
\begin{cases}
-\Delta u=\varphi(u),&\bm x\in\mathbb T,\\
\int_{\mathbb T}u d\bm x=0,
\end{cases}
\end{equation}
where $\varphi\in C^1(\mathbb R)$.
If $\varphi'(s)< \lambda_1$ for any $s\in\mathbb R$, then the steady solution $\bar\omega=-\Delta u$ of the Euler equation is called an \emph{Arnold-stable state}. Due to translation invariance, any  Arnold-stable state must be trivial; see \cite[Proposition 1.1]{CDG}.
In this section, we provide an extension of this result, although it is not directly related to the main theorem of this paper.

\begin{proposition}
Suppose that $u\in C^2(\mathbb T)$ solves \eqref{selp}
with $\varphi\in C^1(\mathbb R)$.
If $\varphi'\leq \lambda_1,$ then $u\in\mathbb E_1.$
\end{proposition}
\begin{proof}
Define $J:\mathring C^1(\mathbb T)\to \mathbb R$ as follows:
\[J(v)=\frac{1}{2}\int_{\mathbb T}|\nabla v|^2 d\bm x-\int_{\mathbb T}\Phi(v)d\bm x, \]
where $\Phi$ is an antiderivative of $\varphi$. Since $\varphi'\leq \lambda_1$, $\Phi$ satisfies
\begin{equation}\label{phi1}
\Phi(s)-\Phi(\tau)\leq \varphi(\tau)(s-\tau)+\frac{1}{2}\lambda_1(s-\tau)^2\quad\forall\,s,\tau\in\mathbb R.
\end{equation}
Based on \eqref{phi1} and applying the Poincar\'e inequality (see \eqref{pocar1}), we compute as follows:
\begin{equation}\label{jvju1}
\begin{split}
&J(v)-J(u)\\
=&\frac{1}{2}\int_{\mathbb T}
|\nabla v|^2-|\nabla u|^2 d\bm x -\int_{\mathbb T}\Phi(v)-\Phi(u)d\bm x\\
\geq& \frac{1}{2}\int_{\mathbb T}|\nabla (u-v)|^2 d\bm x+\int_{\mathbb T}\nabla u\cdot\nabla (v-u)d\bm x-\int_{\mathbb T}\varphi(u)(v-u)+\frac{1}{2}\lambda_1(v-u)^2d\bm x\\
=&\frac{1}{2}\left(\int_{\mathbb T}|\nabla (u-v)|^2 d\bm x-\lambda_1\int_{\mathbb T} (u-v)^2d\bm x\right)\\
\geq& 0,
\end{split}
\end{equation}
and the last inequality is an equality if and only if $u-v\in\mathbb E_1$.
 On the other hand, it is clear that $J$ is invariant under translations, i.e.,
\begin{equation}\label{jvju2}
J(u(\cdot-\bm p))=J(u)\quad\forall\, \bm p\in\mathbb R^2.
\end{equation}
Combining  \eqref{jvju1} and \eqref{jvju2}, we obtain
\[u-u(\cdot-\bm p) \in\mathbb E_1\quad\forall\bm p\in\mathbb  R^2,\]
 which implies that
 \[\partial_{x_i}u\in \mathbb E_1,\quad i=1,2.\]
 Since $\mathbb E_1$ is closed under the operation of taking partial derivatives, we further deduce that
 $-\Delta u\in \mathbb E_1$. Hence
 \[u=\mathsf G(-\Delta u)=\lambda_1^{-1}(-\Delta u)\in\mathbb E_1,\]
which completes the proof.
\end{proof}

\appendix

\section{Some auxiliary results}\label{apdx1}
In this appendix, we present some auxiliary results for the reader's convenience.
We begin with the rigorous definition of the Green operator.
\begin{definition}\label{defap1}
The Green operator $\mathsf{G}$ is defined as the inverse of $-\Delta$ subject to the mean-zero condition, i.e., for any mean-zero function $f$, $u:=\mathsf Gf$ is the unique solution to the following Poisson equation:
\begin{equation}\label{poeq1}
\begin{cases}
  -\Delta u=f, & \bm x\in\mathbb T, \\
 \int_{\mathbb T} u d\bm x=0.
\end{cases}
\end{equation}
\end{definition}

The following  properties of the Green operator are frequently used in this paper.
\begin{itemize}
  \item $\mathsf G$ is a bounded linear operator mapping $\mathring L^p(\mathbb T)$ onto $\mathring W^{2,p}(\mathbb T)$, and thus a compact operator mapping $\mathring L^p(\mathbb T)$ into $\mathring L^q(\mathbb T)$ for any $1\le q\le \infty$. This can be proved by repeating  the argument in the proof of \cite[Lemma 3.1]{CWZ}.
  \item  $\mathsf G$ is symmetric,  i.e.,
\[\int_{\mathbb T}f\mathsf Ggd\bm x=\int_{\mathbb T}g\mathsf Gfd\bm x\quad\forall\,f,g\in \mathring L^p(\mathbb T).\]
This can be proved via integration by parts.
\item $\mathsf G$ is positive definite, i.e.,
\[\int_{\mathbb T}f\mathsf Gf d\bm x\geq 0\quad\forall\,f\in \mathring L^p(\mathbb T),\]
with the inequality being an equality if and only if $f=0$ a.e. on $\mathbb T.$
\end{itemize}

Next, we present two lemmas that are used in deriving the vorticity formulation \eqref{voreq} of the Euler equation.
\begin{lemma}\label{lma1}
The integral of the velocity
is a conserved quantity of the Euler equation.
\end{lemma}

\begin{proof}
Note that the momentum equation \eqref{euler}$_1$ can be written as
\[
\partial_t \bm{v} + \frac{1}{2} \nabla |\bm{v}|^2 - \omega \mathbf{v}^\perp = - \nabla P,
\]
where  $\bm v^\perp=(v_2,-v_1)$.
Integrating the above equation directly yields
\[
\frac{d}{dt} \int_{\mathbb{T}} \bm{v} \, d\bm x = \int_{\mathbb{T}} \omega \bm{v}^\perp \, d\bm x.
\]
On the other hand,
\[
\int_{\mathbb{T}} \omega \mathbf{v}^\perp \, d\bm x
= \int_{\mathbb{T}} (\partial_1 v_2 - \partial_2 v_1)(v_2, -v_1) \, dx
= \int_{\mathbb{T}} (v_2 \partial_1 v_2 - v_2 \partial_2 v_1,\; -v_1 \partial_1 v_2 + v_1 \partial_2 v_1) \, d\bm x.
\]
It is easy to check that
\[
\int_{\mathbb{T}} v_2 \partial_1 v_2 \, d\bm x=\frac{1}{2}\int_{\mathbb T}\partial_1 v_2^2 d\bm x=0,\quad
 \int_{\mathbb{T}} v_1 \partial_{2} v_1 \, d\bm x=\frac{1}{2}\int_{\mathbb T}\partial_2 v_1^2 d\bm x=0,
\]
\[
\int_{\mathbb{T}} v_2 \partial_2 v_1 \, d\bm x
= -\int_{\mathbb{T}} v_1 \partial_2 v_2 \, d\bm x
= \int_{\mathbb{T}} v_1 \partial_1 v_1 \, d\bm x = \frac{1}{2}\int_{\mathbb T}\partial_1 v_1^2 d\bm x=0,
\]
\[
\int_{\mathbb{T}} v_1 \partial_1 v_2 \, d\bm x
= -\int_{\mathbb{T}} v_2 \partial_1 v_1 \, d\bm x
= \int_{\mathbb{T}} v_2 \partial_2 v_2 \, d\bm x = \frac{1}{2}\int_{\mathbb T}\partial_2 v_2^2 d\bm x=0,
\]
where the incompressibility condition $\nabla\cdot\bm v=0$ was used. The proof is complete.
\end{proof}

\begin{lemma}\label{lma2}
Suppose that
\begin{equation}\label{082201}
\int_{\mathbb{T}} \bm{v} \, dx = \bm {0}.
\end{equation}
Then $\psi:=\mathsf G\omega$ satisfies
$
  \bm{v} = (\partial_2 \psi, -\partial_1 \psi).
$
\end{lemma}

\begin{proof}
Observe that
\[-\Delta(\partial_2\psi)=\partial_2(-\Delta\psi)
= \partial_2 \omega = \partial_2(\partial_1 v_2-\partial_2 v_1)=-\Delta v_1.\]
 Since $\partial_2 \psi$, as well as $v_1$ by \eqref{082201}, has zero mean,  it follows that $v_1 = \partial_2 \psi$. Similarly, $v_2 = -\partial_1 \psi$. The proof is complete.
\end{proof}

\section{Characterization of translational orbits in $\mathbb E_1$}\label{apdx2}

\begin{lemma}\label{lemap2}
The translational orbits in $\mathbb E_1$ can be characterized in terms of Fourier coefficients as follows:
\begin{itemize}
\item[(i)]  ($\rm dim(\mathbb E_1)=2$) For any $w, w'\in\mathbb E_1$ with the form
\[w=A\cos(2\pi\bm k\cdot\bm x+\alpha),\quad A\geq0, \ \alpha\in\mathbb R,\]
\[w'=A'\cos(2\pi\bm k\cdot\bm x+\alpha'),\quad A'\geq0, \ \alpha'\in\mathbb R,\]
where $\bm k$ is as in Lemma \ref{lemdime1}(i),  it holds that
\[\mathcal O_{w'}=\mathcal O_{w}\quad \Longleftrightarrow\quad  A=A'.\]
 \item[(ii)] ($\rm dim(\mathbb E_1)=4$)  For any $w, w'\in\mathbb E_1$ with the form
\[w=\sum_{i=1}^2A_i\cos(2\pi\bm k_i\cdot\bm x+\alpha_i),\quad A_i\geq0,\ \alpha_i \in\mathbb R,\]
\[w'=\sum_{i=1}^2A_i'\cos(2\pi\bm k_i\cdot\bm x+\alpha_i'),\quad A_i'\geq0,\ \alpha_i'\in\mathbb R,\]
where $\bm k_1$ and $\bm k_2$ are as in Lemma \ref{lemdime1}(ii),  it holds that
\[\mathcal O_{w'}=\mathcal O_{w}\quad \Longleftrightarrow\quad  (A_1,A_2)=(A_1', A_2').\]

 \item[(iii)]($\rm dim(\mathbb E_1)=6$)  For any $w, w'\in\mathbb E_1$ with the form
\[w=\sum_{i=1}^3A_i\cos(2\pi\bm k_i\cdot\bm x+\alpha_i),\quad A_i\geq0,\ \alpha_i \in\mathbb R,\]
\[w'=\sum_{i=1}^3A_i'\cos(2\pi\bm k_i\cdot\bm x+\alpha_i'),\quad A_i'\geq0,\ \alpha_i'\in\mathbb R,\]
where $\bm k_1, \bm k_2$ and $\bm k_3$ are as in Lemma \ref{lemdime1}(iii), it holds that
\[\mathcal O_{w'}=\mathcal O_{w}\,\, \Longleftrightarrow\,\,  \left(A_1,A_2,A_3,A_1A_2A_3 e^{{\rm i}(\alpha_1+\alpha_2-\alpha_3)}\right)=\left(A_1',A_2',A_3',A_1'A_2'A_3' e^{{\rm i}(\alpha_1'+\alpha_2'-\alpha_3')}\right)
 \]
In particular, if $\alpha_1=\alpha_2=\alpha_1'=\alpha_2'=0,$ then
\[
\mathcal O_{w'}=\mathcal O_{w}\,\, \Longleftrightarrow\,\,
\begin{cases}
\left(A_1, A_2, A_3\right)=
 \left(A_1', A_2', A_3'\right),&\mbox{if }A_1A_2A_3=0,\\
 \left(A_1, A_2, A_3,\cos\alpha_3,\sin\alpha_3\right)=
 \left(A_1', A_2', A_3', \cos\alpha_3',\sin\alpha_3'\right),&\mbox{if }A_1A_2A_3\neq 0.
\end{cases}
\]
\end{itemize}
\end{lemma}

\begin{proof}
 First, we prove (i):
  \begin{align*}
  \mathcal O_{w}=\mathcal O_{w'}&\Longleftrightarrow w'=w(\cdot-\bm p) \,\,\mbox{for some } \bm p\in\mathbb R^2\\
  &\Longleftrightarrow  A'\cos(2\pi\bm k\cdot\bm x+\alpha')\equiv A\cos(2\pi\bm k\cdot(\bm x-\bm p)+\alpha) \,\,\mbox{for some } \bm p\in\mathbb R^2\\
   &\Longleftrightarrow \begin{cases} A'\cos \alpha'=A\cos(\alpha-2\pi\bm k\cdot\bm p),\\
    A'\sin \alpha'=A\sin(\alpha-2\pi\bm k\cdot\bm p)
    \end{cases} \,\,\mbox{for some } \bm p\in\mathbb R^2\\
   &\Longleftrightarrow
    A' e^{{\rm i}\alpha'}=Ae^{{\rm i}(\alpha-2\pi\bm k\cdot\bm p)} \,\,\mbox{for some } \bm p\in\mathbb R^2\\
    &\Longleftrightarrow A'=A.
  \end{align*}

Next, we  prove (ii):
\begin{align*}
  \mathcal O_{w}=\mathcal O_{w'}&\Longleftrightarrow w'=w(\cdot-\bm p) \,\,\mbox{for some } \bm p\in\mathbb R^2\\
  &\Longleftrightarrow \sum_{i=1}^2A_i'\cos(2\pi\bm k_i\cdot\bm x+\alpha_i')\equiv \sum_{i=1}^2A_i\cos(2\pi\bm k_i\cdot(\bm x-\bm p)+\alpha_i) \,\,\mbox{for some } \bm p\in\mathbb R^2\\
   &\Longleftrightarrow
   A_i' e^{{\rm i}\alpha_i'}=A_ie^{{\rm i}(\alpha_i-2\pi\bm k_i\cdot\bm p)},\,i=1,2 \,\,\mbox{for some } \bm p\in\mathbb R^2\\
    &\Longleftrightarrow A_i'=A_i,\,i=1,2.
  \end{align*}
Here we used the fact that there exists a unique $\bm p\in\mathbb R^2$ such that
\[\alpha_1'=\alpha_1-2\pi\bm k_1\cdot\bm p,\quad \alpha_2'=\alpha_2-2\pi\bm k_2\cdot\bm p,\]
since $\bm k_1, \bm k_2$ are linearly independent.

Finally, we prove (iii):
\begin{align*}
  \mathcal O_{w}=\mathcal O_{w'}&\Longleftrightarrow w'=w(\cdot-\bm p) \,\,\mbox{for some } \bm p\in\mathbb R^2\\
  &\Longleftrightarrow \sum_{i=1}^3A_i'\cos(2\pi\bm k_i\cdot\bm x+\alpha_i')\equiv \sum_{i=1}^3A_i\cos(2\pi\bm k_i\cdot(\bm x-\bm p)+\alpha_i) \,\,\mbox{for some } \bm p\in\mathbb R^2\\
   &\Longleftrightarrow
   \begin{cases}
     A_1' e^{{\rm i}\alpha_1'}=A_1e^{{\rm i}(\alpha_1-2\pi\bm k_1\cdot\bm p)}, \\
      A_2' e^{{\rm i}\alpha_2'}=A_2e^{{\rm i}(\alpha_2-2\pi\bm k_2\cdot\bm p)},\\
       A_3' e^{{\rm i}\alpha_3'}=A_3e^{{\rm i}(\alpha_3-2\pi(\bm k_1+\bm k_2)\cdot\bm p)}
   \end{cases}  \,\,\mbox{for some } \bm p\in\mathbb R^2\\
 &\Longleftrightarrow
        \begin{cases}
      A_i'=A_i,\ i=1,2,3,\\
     A_1'A_2'A_3' e^{{\rm i}(\alpha_1'+\alpha_2'-\alpha_3')}= A_1A_2A_3 e^{{\rm i}(\alpha_1+\alpha_2-\alpha_3)},
      \end{cases}
  \end{align*}
  where we have used  $\bm k_3=\bm k_1+\bm k_2.$ To see the last equivalence, notice that the implication $``\Longrightarrow"$ is obvious; while for $``\Longleftarrow"$, we distinguish two cases:
   \begin{itemize}
     \item [(1)]$A_1A_2A_3=0.$ Without loss of generality, we may assume that $A_3=0.$ Then
         \[A_3' e^{{\rm i}\alpha_3'}=A_3e^{{\rm i}(\alpha_3-2\pi(\bm k_1+\bm k_2)\cdot\bm p)}\]
         holds automatically.  Using the fact that $\bm k_1$ and $\bm k_2$ are linearly independent,  there exists some $\bm p\in\mathbb R^2$ such that
      \[\alpha_1'=\alpha_1-2\pi\bm k_1\cdot\bm p\quad\mbox{and}\quad \alpha_2'=\alpha_2-2\pi\bm k_2\cdot\bm p.\]
      Then it follows that
         \[A_1' e^{{\rm i}\alpha_1'}=A_1e^{{\rm i}(\alpha_1-2\pi\bm k_1\cdot\bm p)}\quad \mbox{and}\quad
      A_2' e^{{\rm i}\alpha_2'}=A_2e^{{\rm i}(\alpha_2-2\pi\bm k_2\cdot\bm p)}.\]
     \item [(2)]$A_1A_2A_3\neq 0.$  Using the fact that $\bm k_1$ and $\bm k_2$ are linearly independent again, we can choose  $\bm p\in\mathbb R^2$ such that
      \[\alpha_1'=\alpha_1-2\pi\bm k_1\cdot\bm p\quad\mbox{and}\quad \alpha_2'=\alpha_2-2\pi\bm k_2\cdot\bm p.\]
      Then   \[A_1' e^{{\rm i}\alpha_1'}=A_1e^{{\rm i}(\alpha_1-2\pi\bm k_1\cdot\bm p)}\quad \mbox{and}\quad
      A_2' e^{{\rm i}\alpha_2'}=A_2e^{{\rm i}(\alpha_2-2\pi\bm k_2\cdot\bm p)}.\]
      Moreover, in view of the condition
        \[A_1'A_2'A_3' e^{{\rm i}(\alpha_1'+\alpha_2'-\alpha_3')}= A_1A_2A_3 e^{{\rm i}(\alpha_1+\alpha_2-\alpha_3)},\]
        such $\bm p$ also ensures that
        \[A_3' e^{{\rm i}\alpha_3'}=A_3e^{{\rm i}(\alpha_3-2\pi(\bm k_1+\bm k_2)\cdot\bm p)}.\]
   \end{itemize}
\end{proof}

\section{On a system of polynomial equations}\label{apdx3}

\begin{lemma}\label{lemap3}
The following system of polynomial equations has at most 6 solutions for any $c_1,c_2,c_3\in\mathbb R$:
\begin{equation}\label{ase0}
\begin{cases}
 x+ y+ z=c_1,\\
 x^2+ y^2+ z^2+4\left(x y+xz+yz\right)=c_2,\\
  x^3+ y^3+ z^3+9\left(x^2y+x^2z+xy^2+y^2z+xz^2+yz^2\right)+27 x y z=c_3,
 \end{cases}
\end{equation}
\end{lemma}

\begin{proof}
Notice that  \eqref{ase0}$_2$ can be written as
\begin{equation}\label{xyz01}
x^2+\left( y+ z\right)^2+2 y z+4 x\left( y+ z\right)=c_2.
\end{equation}
Inserting  \eqref{ase0}$_1$ into \eqref{xyz01}, we obtain
\[
 x^2+\left(c_1- x\right)^2+2 y z+4 x\left(c_1- x\right)=c_2,
\]
which implies that
\begin{equation}\label{ase2}
 y z= x^2-c_1 x
+\frac{1}{2}(c_2-c_1^2).
\end{equation}
To proceed, notice that   \eqref{ase0}$_3$ can be written as
\begin{equation}\label{ase3}
   x^3+\left( y+ z\right)^3+9 x^2\left( y+ z\right)+9 x\left(( y+ z)^2-2 y z\right)+6 y z\left( y+ z\right)+27 x y z=c_3.
\end{equation}
Inserting \eqref{ase0}$_1$ and \eqref{ase2} into \eqref{ase3} gives
\begin{equation}\label{ase4}
\begin{split}
 & x^3+(c_1- x)^3+9 x^2(c_1- x)+9 x\left((c_1- x)^2-2\left( x^2-c_1 x
+\frac{1}{2}(c_2-c_1^2)\right)\right)\\
&+6\left( x^2-c_1 x
+\frac{1}{2}(c_2-c_1^2)\right)\left(c_1- x\right)+27 x\left( x^2-c_1 x
+\frac{1}{2}(c_2-c_1^2)\right)=c_3,
\end{split}
  \end{equation}
which can be simplified as
\[
3 x^3-3c_1 x^2+\frac{3}{2}\left(c_2-c_1^2\right) x +3c_1c_2-2 c_1^3-c_3=0.
\]
This is a cubic equation, which has at most 3 roots. Moreover, for each root $x$, there are at most 2 pairs $(y,z)$ satisfying  \eqref{ase0}$_1$ and \eqref{ase2}. Therefore, \eqref{ase0}  has at most 6 solutions. The proof is complete.

\end{proof}

\bigskip

 \noindent{\bf Acknowledgements:} G. Wang was supported by NNSF of China (Grant No. 12471101) and  Fundamental Research Funds for the Central Universities (Grant No. DUT23RC(3)077).

\bigskip
\noindent{\bf  Data Availability} Data sharing not applicable to this article as no datasets were generated or analysed during the current study.

\bigskip
\noindent{\bf Conflict of interest}    The author declares that he has no conflict of interest to this work.

\phantom{s}
 \thispagestyle{empty}

\end{document}